\documentclass[12pt]{article}
\usepackage{amsmath,amssymb,latexsym}
\usepackage{enumerate}
\usepackage{amsthm}
\usepackage{graphicx}
\usepackage{epsfig}

\newcommand\NN{\mathrm{I\!N}}

\newcommand\CC{\mathbb{C}}

\newtheorem{theorem}{Theorem}[section]
\newtheorem{definition}[theorem]{Definition}
\newtheorem{corollary}[theorem]{Corollary}
\newtheorem{example}[theorem]{Example}
\newtheorem{lemma}[theorem]{Lemma}
\newtheorem{remark}[theorem]{Remark}
\newtheorem{proposition}[theorem]{Proposition}

\newcommand{\address}{Address: Department of Mathematics, University of North Texas, 1155 Union Circle \#311430, Denton, TX 76203-5017, USA; E-mail: kiko.kawamura@unt.edu, andrew.allen@unt.edu}

\numberwithin{equation}{section}

\title{Revolving Fractals}

\author{Kiko Kawamura and Andrew Allen \\University of North Texas \footnote{\address}}

\begin{document}

\maketitle

\begin{abstract}
Davis and Knuth in 1970 introduced the notion of revolving sequences, as representations of a Gaussian integer. Later, Mizutani and Ito pointed out a close relationship between a set of points determined by all revolving sequences and a self-similar set, which is called the Dragon. We will show how their result can be generalized, giving new parametrized expressions for certain self-similar sets. 
\end{abstract}

\section*{Introduction}

In 1970, C.~Davis and D.~E.~Knuth~\cite{Davis+Knuth-1970} introduced the notation of {\it revolving representations} of a Gaussian integer: for any $z=x+iy$ with $x, y \in \mathbb{Z}$,  there exists a revolving sequence $(\delta_{0}, \delta_{1},\dots \delta_{n})$ such that  
\begin{equation*}
  z=\sum_{k=0}^{n} \delta_{n-k}(1+i)^{k}, 
\end{equation*}
where $\delta_{k} \in \{0, 1, -1, i, -i\}$ with the restriction that the non-zero values must follow the cyclic pattern from left to right:
$$\cdots \to 1 \to (-i) \to (-1) \to i \to 1 \to \cdots$$

For instance, they gave the following example:
$$-5+33i = (1 \; 0 \; 0 \; 0 \; (-i) \; (-1) \; i \; 1 \; 0 \; (-i) \; 0)_{1+i}.$$

They also showed that each Gaussian integer has exactly four representations of this type: one each in which the right-most non-zero value takes on the values $1, -1, i, -i$.

\begin{figure}
\begin{center}
\includegraphics[height=4.5cm,width=4.5cm, bb=0 0 400 400]{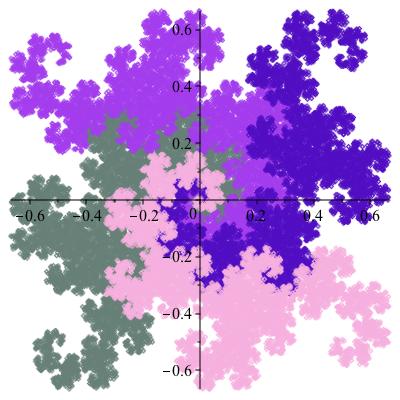} 
\qquad 
\includegraphics[height=4.5cm,width=4.5cm, bb=0 0 400 400]{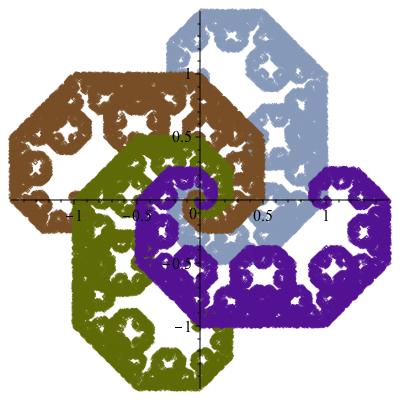}
\end{center}
\caption{$X$ (left) and $X^{*}$ (right)}
\label{figure1}
\end{figure}
\medskip

Let $W$ be the set of all revolving sequences, and define the set 
\begin{equation*}
  X:=\left\{ \sum_{n=1}^{\infty} \delta_{n}(1+i)^{-n}: (\delta_{1}, \delta_{2}, \delta_{3},\dots )\in W \right\}.
\end{equation*}
Notice that each revolving sequence determines a complex number and $X$ is a set of points in the complex plane. The set $X$ is shown in the left half of Figure \ref{figure1}. Mizutani and Ito~\cite{Ito-1987} proved the following theorem using techniques from symbolic dynamics:

\begin{theorem}[Mizutani-Ito, 1987]~\\ 
(i) The set $X$ is tiled by four Dragons $\{D_k, k=0,1,2,3\}$, that is 
\begin{equation*}
X=\bigcup_{k=0}^{3}D_{k}=\bigcup_{k=0}^{3} i^k D, 
\end{equation*}
where $D=\psi_1(D)\cup\psi_2(D)$ is the self-similar set generated by
\begin{equation*}
\label{rev-dragon}
\begin{cases}
 \psi_1(z)=(\frac{1-i}{2})z,& \\
 \psi_2(z)=(\frac{-1-i}{2})z + \frac{1-i}{2}.&
\end{cases}
\end{equation*}
~\\
(ii) \begin{equation*}
\lambda(D_k \cap D_{k^{'}})=0, \mbox{ for each $k \neq k^{'}.$}
\end{equation*} 
\end{theorem}

In the same paper, they mentioned an interesting question. Define another set $X^{*}$ by 
\begin{equation*}
X^{*}:=\left\{ \sum_{n=1}^{\infty} \overline{\delta_{n}}(1+i)^{-n}: (\delta_{1},\delta_{2}, \delta_{3}, \dots )\in W \right\}.
\end{equation*}

Notice that $\overline{\delta_n}$ moves on the unit circle counterclockwise instead of clockwise.  
The set $X^{*}$ is shown in the right half of Figure \ref{figure1}. Computer simulations suggested to Mizutani and Ito that $X^{*}$ must be a union of four L\'evy's curves; however, they could not give a mathematical proof.

\begin{figure}
  	\begin{center}
	\epsfig{file=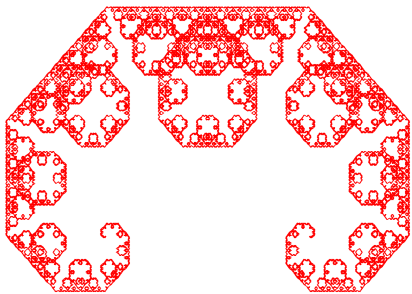, height=1.75in, width=.51\textwidth} 
	\end{center}
  \caption{L\'evy's curve}
	\label{dragon}
\end{figure}

Recall that L\'evy's curve is a continuous curve with positive area. It was introduced by Paul L\'evy in 1939~\cite{Levy-1939}. 
Figure \ref{dragon} shows the graph of L\'evy's curve, which is a self-similar set $L=\phi_1(L)\cup\phi_2(L)$ generated by the similar contractions 
\begin{equation}
\label{eq:levy}
\begin{cases}
 \phi_1(z)=(\frac{1+i}{2}) z,& \\
 \phi_2(z)=(\frac{1-i}{2})z + \frac{1+i}{2}.&
\end{cases}
\end{equation}
\bigskip

Kawamura in 2002~\cite{Kawamura-2002} finally gave a proof of Mizutani and Ito's conjecture. 

\begin{theorem}[Kawamura, 2002]~\\ 
The conjugate of $X^{*}$ is a union of four copies of L\'evy's curves $L$ generated by \eqref{eq:levy}, that is 
\begin{equation*} 
\overline{X^{*}}=\bigcup_{k=0}^{3} i^k L.
\end{equation*}
\end{theorem}

It is worth mentioning that the proof is completely different from Mizutani and Ito's approach. Instead of using technique from symbolic dynamics, Kawamura considered the following functional equation 
\begin{equation}
\label{eq:kiko}
        f_{\alpha, \gamma}(x)=
        \begin{cases}
                \alpha f_{\alpha,\gamma}(2x),
                        & \qquad 0 \leq x < 1/2, \\
                \gamma f_{\alpha,\gamma}(2x-1)+(1-\gamma),
                        & \qquad 1/2 \leq x \leq 1,
       \end{cases}
\end{equation}
where $\alpha$ and $\gamma$ are complex parameters satisfying $|\alpha|<1$, $|\gamma|<1$. She proved the existence of a unique bounded solution $f_{\alpha,\gamma}:[0,1] \to \mathbb{C}$ of \eqref{eq:kiko} and gave the explicit expression. Observe that the closure of the image of this bounded solution $f_{\alpha,\gamma}([0,1])$ is a self-similar set generated by two contractions $\phi_1(z)=\alpha z$ and $\phi_2(z)=\gamma z + (1-\gamma)$. In particular, if $\alpha=(1+i)/2$ and $\gamma=(1-i)/2$, $f_{\alpha,\gamma}([0,1])=L$.

\bigskip

L\'evy's curve and Dragon are very different: one is a continuous curve while the other is a tiling fractal; however, both are self-similar sets. Thus, the following questions arise naturally. 
\begin{enumerate}
\item Is there a generalized relationship between sets of revolving sequences and self-similar sets? In particular, we are interested in describing self-similar sets which arise from more general revolving sequences, where the $90$ degree angle of rotation is replaced with a more general angle. 
\item Is there a simpler way to prove both Mizutani-Ito and Kawamura's theorems?
\end{enumerate}

\section{Generalized Revolving Sequences}

Before stating our results, some notation need to be introduced. Let $\alpha \in \CC$ denote a complex parameter satisfying $|\alpha|<1$. Let $\theta$ be an angle with $-\pi < \theta \leq \pi$ and a rational multiple of $2 \pi$. More precisely, there are $p \in \NN, q \in \NN_{0}$ such that $|\theta|=\frac{2 \pi q}{p}$. Define 
$$\Delta_{\theta}:= \{0, 1, e^{i \theta}, e^{2i \theta}, \cdots e^{(p-1)i \theta}\}.$$

\begin{definition}
A sequence $(\delta_{1},\delta_{2},\dots)\in\Delta_{\theta}^{\NN}$ satisfies the {\em Generalized Revolving Condition (GRC)}, if the  subsequence obtained after the removal of its zero elements is a (finite or infinite) truncation of the sequence $(e^{i \theta})$. More precisely, let $(n_i):=\{n:\delta_n \neq 0\}$. Then, 
$\delta_{n_{i+1}}=e^{i \theta} \delta_{n_i}.$
\end{definition}
Notice that $\delta_n$ moves on the unit circle counterclockwise if $\theta >0$, and clockwise if $\theta < 0$. 
\bigskip

Define $W_{\theta}$ as the set of all generalized revolving sequences with parameter $\theta$: 
\begin{equation*}
W_{\theta}:=\{(\delta_1, \delta_2, \cdots) \in \Delta_{\theta}^{\NN}: (\delta_1, \delta_2, \cdots) \mbox{ satisfies the GRC}\},
\end{equation*}
and for a given $\alpha \in \CC$ such that $|\alpha|<1$, define 
\begin{equation*}
  X_{\alpha,\theta}:=\left\{ \sum_{n=1}^{\infty} \delta_{n}\alpha^{n}: (\delta_{1}, \delta_{2}, \delta_{3},\dots )\in W_{\theta} \right\}. \label{xalpha1}
\end{equation*} 

Notice that each generalized revolving sequence determines a complex number and $X_{\alpha,\theta}$ is a set of points in the complex plane. Two examples of $X_{\alpha,\theta}$ are shown in Figure \ref{figure2}. It is not hard to imagine that $X_{\alpha,\theta}$ is a union of self-similar sets; however, it is not immediately clear which iterated function system generates these self-similar sets. 
 
\begin{figure}
\begin{center}
\includegraphics[height=5cm,width=5cm, bb=0 0 400 400]{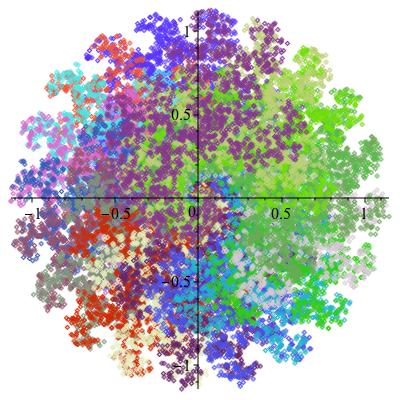} 
\qquad
\includegraphics[height=5cm,width=5cm, bb=0 0 400 400]{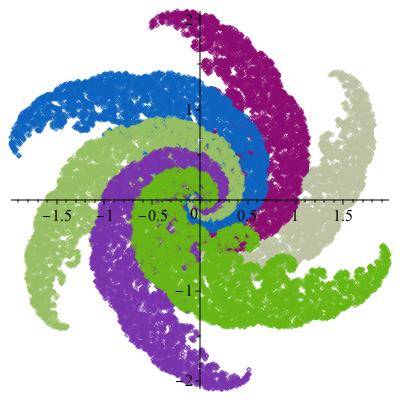}
\end{center}
\caption{$X_{\alpha,\theta}: (\alpha, \theta)=(\frac{1+i}{2}, \frac{\pi}{10})$ (left) and $(\alpha, \theta)=(\frac{1-i}{2}, \frac{\pi}{3})$ (right)}
\label{figure2}
\end{figure}
\medskip

Using a direct approach different from both \cite{Ito-1987} and \cite{Kawamura-2002}, we obtain the following theorem.

\begin{theorem}
\label{th:main1}
$X_{\alpha, \theta}$ is a union of $p$ copies of $K_{\alpha,\theta}$:
\begin{equation*}
X_{\alpha,\theta}=\bigcup_{l=0}^{p-1}(e^{i \theta})^l K_{\alpha,\theta},
\end{equation*}
where $K_{\alpha,\theta}=\psi_1(K_{\alpha,\theta})\cup\psi_2(K_{\alpha,\theta})$ is the self-similar set generated by the iterated function system (IFS): 
\begin{equation}
\label{eq:IFS1} 
\begin{cases}
 \psi_1(z)=\alpha z,& \\
 \psi_2(z)=(\alpha e^{i \theta}) z+\alpha.
\end{cases}
\end{equation}
\end{theorem}

\begin{remark}
{\rm 
Using a similar approach as in~\cite{Kawamura-2002}, Young~\cite{Young-2015} essentially found the result of Theorem~\ref{th:main1} under the RTG Undergraduate Summer Research Program; however, his proof was incomplete.
}
\end{remark} 

\begin{example}
{\rm 
Both Mizutani-Ito's and Kawamura's results are included in this setting as special cases. It is clear that $X_{\alpha,\theta}$ is a union of Dragons if $\alpha=\frac{1-i}{2}$ and $\theta=-\pi/2$. If $\alpha=\frac{1-i}{2}$ and $\theta=\pi/2$, then $X_{\alpha,\theta}$ is a union of L\'evy's curves generated by
\begin{equation}
\label{rev-levy}
\begin{cases}
 \psi_1(z)=(\frac{1-i}{2})z,& \\
 \psi_2(z)=(\frac{1+i}{2})z + \frac{1-i}{2}.&
\end{cases}
\end{equation}
}
\end{example}

Notice that \eqref{rev-levy} is different from \eqref{eq:levy}. Let $P$ be the self-similar set $P=\psi_1(P)\cup\psi_2(P)$ generated by \eqref{rev-levy}. It is clear that $\overline{L}=P$ since 
$$\overline{L}=\overline{\phi_1(L)}\cup \overline{\phi_2(L)}=\psi_1(\overline{L})\cup\psi_2(\overline{L}).$$ 

Recall the celebrated theorem of Hutchinson~\cite{Hutchinson-1981}: For any finite family of similar contractions $\psi_{1},\psi_{2},\dots,\psi_{m}$ on $\mathbb{R}^{n}$, there exists a unique self-similar set $X \subset \mathbb{R}^{n}$, which is a unique non-empty compact solution of the set equation $X= \psi_{1}(X) \cup \psi_{2}(X) \cup \dots \cup \psi_{m}(X)$. However, the converse is not true. In fact, a self-similar set can be constructed by many different families of similar contractions. 
\medskip

One of the challenges of this type of question is to find a suitable pair of contractions which matches the position of the set $X_{\alpha,\theta}$ exactly. Once the suitable pair of contractions is found, a more direct proof is possible, using the following lemma and proposition.

\medskip

Define a subset of $X_{\alpha, \theta}$ as follows.
\begin{equation}
  X_{1, \alpha,\theta}:=\left\{ \sum_{n=1}^{\infty} \delta_{n}\alpha^{n}: \delta_{j_1}=1, (\delta_{1}, \delta_{2}, \delta_{3},\dots )\in W_{\theta} \right\},
\label{eq:X1}
\end{equation}
where $j_{1}:=\min\{j: \delta_{j} \not= 0\}$.

\begin{lemma} \label{lemma:X1-closed}
$X_{1,\alpha,\theta}$ is a closed set.
\end{lemma}

\begin{proof}
Let $(x_k)$ be a sequence in $X_{1,\alpha,\theta}$, converging to some point $x$. 
For each $k \in \NN$, there exists a sequence 
$(\delta^{k}_n) \in W_{\theta}$ such that $x_k=\sum_{n=1}^{\infty} \delta^{k}_{n}\alpha^{n}$. 

It suffices to construct a sequence $(\delta_n)$ such that for each $n \geq 1$, there exists $k \in \NN$ such that $(\delta_n)$ and $(\delta^{k}_n)$ start with the same initial word of length $n$. Obviously, then $x=\sum_{n=1}^{\infty} \delta_{n}\alpha^{n} \in X_{1,\alpha,\theta}$.

A suitable sequence $(\delta_n)$ can be constructed by induction as follows. First, choose $\delta_1 \in \{0,1\}$ so that there exist infinitely many $k \in \NN$ such that $\delta^{k}_1=\delta_1$. Next, suppose an initial word $\delta_1, \delta_2, \dots, \delta_n$ has been defined for some $n \geq 1$ so that $\delta^k_1\delta^k_2\dots \delta^k_n=\delta_1\delta_2\dots \delta_n$ for infinitely many $k\in\NN$. Then choose $\delta_{n+1} \in \{0, e^{i\theta}\delta_{j_0(n)}\}$, where $j_0(n):=\max\{j \leq n: \delta_{j} \not= 0\}$, in such a way that there exist infinitely many $k\in\NN$ such that
$$\delta^k_1\delta^k_2\dots \delta^k_n\delta^k_{n+1}=\delta_1\delta_2\dots \delta_n\delta_{n+1}.$$
This gives the desired sequence $(\delta_n)$.
\end{proof}

\begin{corollary}\label{corollary:X-closed}
$X_{\alpha,\theta}$ is a closed set.
\end{corollary}

\begin{proof}
Notice that 
\begin{equation*}
X_{\alpha,\theta}=\bigcup_{l=0}^{p-1}(e^{i \theta})^l X_{1, \alpha,\theta}.
\end{equation*}
Using the fact that the union of finitely many closed set is closed, $X_{\alpha,\theta}$ is closed.
\end{proof}

\begin{proposition} \label{prop-X1}
$X_{1, \alpha,\theta}$ satisfies the set equation: 
$$X_{1, \alpha,\theta}= \psi_{1}(X_{1, \alpha,\theta}) \cup \psi_{2}(X_{1, \alpha,\theta}), $$
where $\{\psi_1,\psi_2\}$ is the IFS from \eqref{eq:IFS1}.
\end{proposition}

\begin{proof}
Let $x=\sum_{n=1}^{\infty}\delta_n \alpha^{n}\in X_{1, \alpha, \theta}$. If $\delta_1=0$, set $\delta_j':=\delta_{j+1}$ for $j=1,2,\dots$. Then $(\delta_j')$ satisfies the generalized revolving condition with its first nonzero digit equal to 1, so 
$$x=\alpha \sum_{j=1}^{\infty}\delta_j' \alpha^{j}\in \psi_1(X_{1, \alpha, \theta}).$$
If $\delta_1=1$, set $\delta_j':=e^{-i\theta}\delta_{j+1}$ for $j=1,2,\dots$. Since the second nonzero digit of $(\delta_n)$ is $e^{i\theta}$, the sequence $(\delta_j')$ satisfies the generalized revolving condition with its first nonzero digit equal to 1, so 
$$x=\alpha+\alpha e^{i\theta}\sum_{j=1}^\infty\delta_j' \alpha^j \in \psi_2(X_{1, \alpha, \theta}).$$
Thus $X_{1, \alpha,\theta}\subset \psi_{1}(X_{1, \alpha,\theta}) \cup \psi_{2}(X_{1, \alpha,\theta})$. 

\medskip
The reverse inclusion follows analogously. Let $x \in \psi_{1}(X_{1, \alpha,\theta}) \cup \psi_{2}(X_{1, \alpha,\theta})$. 
If $x=\alpha \sum_{n=1}^{\infty}\delta_n \alpha^{n}\in \psi_1(X_{1, \alpha, \theta})$, set 
\begin{equation*} 
\delta_{j}':=
\begin{cases}
 0,& \qquad \mbox{if } j=1, \\
 \delta_{j-1},& \qquad \mbox{if } j \geq 2.
\end{cases}
\end{equation*}
Then $(\delta_j')$ satisfies the generalized revolving condition with its first nonzero digit equal to 1, so 
$$x=\sum_{n=2}^{\infty}\delta_{n-1} \alpha^{n}=\sum_{j=1}^{\infty}\delta_j' \alpha^{j} \in X_{1, \alpha, \theta}.$$
If $x=(\alpha e^{i \theta})\sum_{n=1}^{\infty}\delta_n \alpha^{n} + \alpha \in \psi_2(X_{1, \alpha, \theta})$, set 
\begin{equation*} 
\delta_{j}':=
\begin{cases}
 1,& \qquad \mbox{if } j=1, \\
 e^{i\theta}\delta_{j-1},& \qquad \mbox{if } j \geq 2.
\end{cases}
\end{equation*}
Since the second nonzero digit of $(\delta_{j}')$ is $e^{i\theta}$, the sequence $(\delta_j')$ satisfies the generalized revolving condition with its first nonzero digit equal to 1, so 
\begin{equation*}
x=\sum_{n=2}^{\infty}(e^{i \theta} \delta_{n-1}) \alpha^{n} + \alpha=\sum_{j=1}^{\infty}\delta_{j}'\alpha^{j} \in  X_{1, \alpha, \theta}.
\end{equation*}
Thus, $\psi_{1}(X_{1, \alpha,\theta}) \cup \psi_{2}(X_{1, \alpha,\theta}) \subset X_{1, \alpha,\theta}$.
\end{proof}

\begin{proof}[Proof of Theorem \ref{th:main1}]~\\
Since the set equation $X=\psi_1(X)\cup \psi_2(X)$ has a unique nonempty compact solution, Theorem \ref{th:main1} follows immediately from Lemma~\ref{lemma:X1-closed} and Proposition~\ref{prop-X1}.
\end{proof}

\section{Signed Revolving Sequences}

Theorem \ref{th:main1} shows a direct relationship between generalized revolving sequences and self-similar sets generated by the IFS from \eqref{eq:IFS1}:
\begin{equation*} 
\begin{cases}
 \psi_1(z)=\alpha z,& \\
 \psi_2(z)=(\alpha e^{i \theta}) z+\alpha.&
\end{cases}
\end{equation*}

Many interesting self-similar sets are generated by \eqref{eq:IFS1}; however, Koch's curve, a famous self-similar set,  is not generated by \eqref{eq:IFS1} but by a different pair of two similar contractions: 
\begin{align}
\label{eq:second case}
 \begin{cases}
 \psi_1(z)=\alpha \overline{z},& \\
 \psi_2(z)=(\alpha e^{i \theta}) \overline{z}+\alpha.  &
 \end{cases}
\end{align}
In particular, if $\alpha=1/2+(\sqrt{3}/6)i, \theta=-\pi/3$, the IFS \eqref{eq:second case} generates Koch's curve.

\medskip

A reversed question arises naturally: What kind of revolving sequences are related to self-similar sets generated by the IFS \eqref{eq:second case}? More precisely,  given the attractor $K^{2}_{\alpha,\theta}$ of the IFS \eqref{eq:second case}, we want to find a suitable set of ``revolving" sequences such that the analog of the set $X_{\alpha,\theta}$ from Section 1 is  
$$\bigcup_{l=0}^{p-1}(e^{i \theta})^l K^{2}_{\alpha,\theta}.$$

Recall that $\alpha \in \CC$ is a complex parameter satisfying $|\alpha|<1$ and $\theta$ be an angle with $|\theta|=\frac{2 \pi q}{p}$ where  
$p \in \NN, q \in \NN_{0}$. The generalized revolving sequences from Section 1 always follow a fixed direction on the unit circle, depending on the given $\theta$. How does the introduction of complex conjugates in the IFS influence the corresponding type of revolving sequences?


\begin{definition} \label{def:SRC}
A sequence ${(\delta_{1},\delta_{2},\dots)\in\Delta_{\theta}^{\NN}}$ satisfies the {\em Signed Revolving Condition (SRC)}, if  
\begin{enumerate}
\item $\delta_1$ is free to choose, 
\item If $\delta_{1}=\delta_{2}=\dots =\delta_{k}=0$, then  $\delta_{k+1}$ is free to choose,
\item Otherwise, $\delta_{k+1}=0$ or \begin{align*}
\delta_{k+1}=\begin{cases} 	
(e^{+i\theta})\delta_{j_{0}(k)}, &\mbox{if $j_0(k)$ is odd},\\ 
(e^{-i\theta})\delta_{j_{0}(k)}, &\mbox{if $j_0(k)$ is even}, 
\end{cases}
\end{align*}
where $j_{0}(k):=\max \{j \le k : \delta_{j} \not= 0 \}$. 
\end{enumerate} 
\end{definition}
Roughly speaking, $j_0(k)$ is the last time before time $k$ that $\delta_j$ is on the unit circle. 

\medskip

Notice that $\delta_n$ is either zero or lies on the unit circle, and its direction of motion (that is, where it moves to at time $n$) depends on the last time $j<n$ when $\delta_j$ is on the unit circle. If the last visit to the unit circle happened at an even time, then $\delta_{n}$ moves clockwise along the circle. On the other hand, if the last visit to the unit circle happened at an even time, then $\delta_n$ moves counterclockwise along the circle. For example, 
$$0 \to 1 \to e^{-i\theta} \to 0 \to 1 \to e^{i\theta} \to 0 \to 0 \to 1 \to \ldots.$$

Compared to the generalized revolving sequences from Section 1, which always move in the same direction, we see that the direction of movement of the sequence $(\delta_n)$ depends on its past.

\bigskip

Define $W^{\pm}_{\theta}$ as the set of all signed revolving sequences with parameter $\theta$: 
$$W^{\pm}_{\theta}:=\{(\delta_{1},\delta_{2},\dots )\in \Delta_{\theta}^{\NN}:(\delta_{1},\delta_{2},\dots \delta_k, \dots)\mbox{ satisfies the SRC} \},$$ 
and for a given $\alpha \in \CC$ such that $|\alpha|<1$, define 
\begin{equation*}
  X^2_{\alpha,\theta}:=\left\{ \sum_{n=1}^{\infty} \delta_{n} \prod_{j=1}^{n}\eta_j 
: (\delta_{1}, \delta_{2},\dots )\in W^{\pm}_{\theta} \right\},
\end{equation*}
where $\eta_1=\alpha$ and $\eta_{j+1}=\overline{\eta_j}$ for $j=1,2, \dots$. Four examples of $X^2_{\alpha, \theta}$ are shown in Figure~\ref{figure3}. 

\begin{figure}
\begin{center}
\includegraphics[height=5.5cm,width=5.5cm, bb=0 0 400 400]{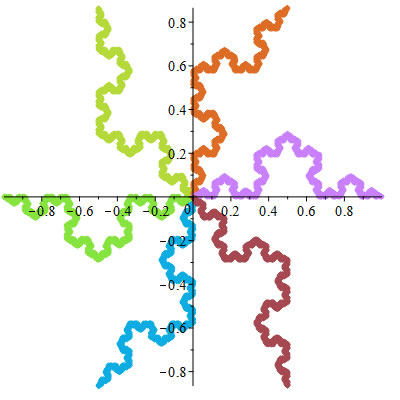} 
\qquad 
\includegraphics[height=5.5cm,width=5.5cm, bb=0 0 400 400]{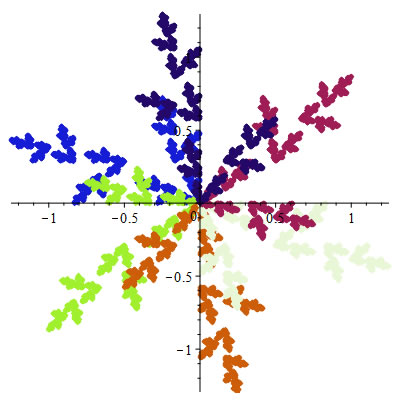} \\

\includegraphics[height=5.5cm,width=5.5cm, bb=0 0 400 400]{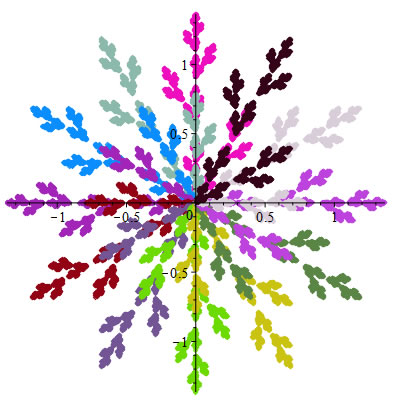} 
\qquad 
\includegraphics[height=5.5cm,width=5.5cm, bb=0 0 400 400]{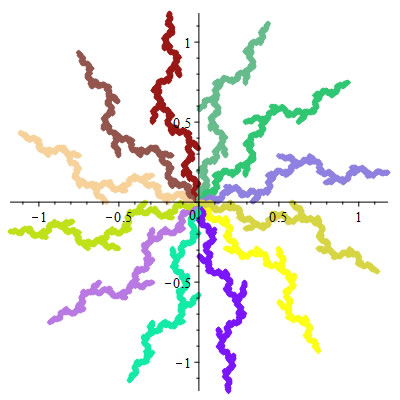}
\end{center}
\caption{$X^{2}_{\alpha,\theta}: (\alpha, \theta)=(\frac{1}{2}+\frac{\sqrt{3}i}{6}, -\frac{\pi}{3})$ (top left), $(\alpha, \theta)=(\frac{1}{2}+\frac{\sqrt{3}i}{6}, \frac{\pi}{3})$ (top right), 
$(\alpha, \theta)=(\frac{1}{2}+\frac{\sqrt{3}i}{6}, \frac{\pi}{6})$ (bottom left), $(\alpha, \theta)=(\frac{1}{2}+\frac{\sqrt{3}i}{6}, -\frac{\pi}{6})$ (bottom right)}
\label{figure3}
\end{figure}

\medskip

Let $j_{1}:=\min\{j: \delta_{j} \not= 0\}$ and define a subset of $X^{2}_{\alpha, \theta}$ as follows.
\begin{equation}
  X^{2}_{1, \alpha,\theta}=\left\{\sum_{n=1}^{\infty} \delta_{n} \prod_{j=1}^{n}\eta_j : \delta_{j_1}=1, (\delta_{1}, \delta_{2},\dots )\in W^{\pm}_{\theta} \right\}. \label{xalpha2}
\end{equation} 

A straightforward modification of the proof of Lemma~\ref{lemma:X1-closed} gives the following Lemma~\ref{lemma:X2-closed} and Corollary~\ref{corollary:X2-closed}. 

\begin{lemma} \label{lemma:X2-closed}
$X^{2}_{1,\alpha,\theta}$ is a closed set.
\end{lemma}

\begin{corollary}\label{corollary:X2-closed}
$X^{2}_{\alpha,\theta}$ is a closed set.
\end{corollary}

\begin{proposition}
\label{prop:X2-set-equation}
$X^{2}_{1, \alpha,\theta}$ satisfies the set equation 
$$X^{2}_{1, \alpha,\theta}= \psi_{1}(X^{2}_{1, \alpha,\theta}) \cup \psi_{2}(X^{2}_{1, \alpha,\theta}),$$
where $\{\psi_1,\psi_2\}$ is the IFS from \eqref{eq:second case}.
\end{proposition}

\begin{proof}
Let $x=\sum_{n=1}^{\infty}\delta_n \prod_{l=1}^{n}\eta_l \in X^{2}_{1, \alpha, \theta}$. If $\delta_1=0$, set $\delta_j':=\overline{\delta_{j+1}}$ for $j=1,2,\dots$. Then $(\delta_j')$ satisfies the signed revolving condition with its first nonzero digit equal to 1, so 
$$x=\alpha \sum_{j=1}^{\infty}\overline{\delta_j'} \prod_{l=1}^{j}\overline{\eta_l}
=\alpha\overline{\sum_{j=1}^{\infty}\delta_j'\prod_{l=1}^{j}\eta_l}\in \psi_1(X^{2}_{1, \alpha, \theta}).$$

If $\delta_1=1$, set $\delta_j':=e^{i\theta}\overline{\delta_{j+1}}$ for $j=1,2,\dots$. Since the second nonzero digit of $(\delta_n)$ is $e^{i\theta}$, the sequence $(\delta_j')$ satisfies the signed revolving condition with its first nonzero digit equal to 1, so 
$$x=\alpha+\alpha e^{i\theta}\sum_{j=1}^\infty\overline{\delta_j'} \prod_{l=1}^{j}\overline{\eta_l} \in \psi_2(X^{2}_{1, \alpha, \theta}).$$
Thus $X^{2}_{1, \alpha,\theta}\subset \psi_{1}(X^{2}_{1, \alpha,\theta}) \cup \psi_{2}(X^{2}_{1, \alpha,\theta})$. The reverse inclusion follows analogously.
\end{proof}

Since the set equation $X=\psi_1(X)\cup \psi_2(X)$ has a unique nonempty compact solution, Theorem \ref{th:main2} follows immediately from Lemma~\ref{lemma:X2-closed} and Proposition \ref{prop:X2-set-equation}.

\begin{theorem}
\label{th:main2}

Let $K^{2}_{\alpha,\theta}$ be the self-similar set generated by the IFS from \eqref{eq:second case}. Then, 
\begin{equation*}
\bigcup_{k=0}^{p-1}(e^{i \theta})^k K^{2}_{\alpha,\theta}=X^2_{\alpha, \theta}.
\end{equation*}
\end{theorem}

\begin{remark}
{\rm
It is interesting to note that, while $K^{2}_{\alpha,\theta}$ is the attractor of an autonomous IFS (where the maps applied at each step do not change), its representation by a set of revolving sequences involves a rule that is past-dependent.
}
\end{remark}

\section{Alternating Sequences}

Both Propositions~\ref{prop-X1} and \ref{prop:X2-set-equation} gave more explicit description of certain self-similar sets generated by the IFS from \eqref{eq:IFS1} and \eqref{eq:second case} respectively. In these two iterated function systems, either both maps or neither involve a reflection. But what happens if exactly one of the maps includes a reflection? For example, what kind of revolving sequences are related to self-similar sets generated by the IFS
\begin{align}
\label{eq:third case}
 \begin{cases}
 \psi_1(z)=\alpha z,& \\
 \psi_2(z)=(\alpha e^{i \theta})\overline{z}+\alpha, &
 \end{cases}
\end{align}
where $\alpha \in \CC$ such that $|\alpha|<1$ and $|\theta|=|\frac{2 \pi q}{p}| \leq \pi$ ?

(Notice that the self-similar sets generated by 
\begin{align*}
 \begin{cases}
 \psi_1(z)=\alpha \overline{z},& \\
 \psi_2(z)=(\alpha e^{i \theta})z+\alpha, &
 \end{cases}
\end{align*}
are essentially the same as those generated by \eqref{eq:third case}, so this fourth case does not require separate treatment.)

As in Section 2, we want to find a suitable set of ``revolving" sequences $X^{3}_{1, \alpha,\theta}$ satisfying the set equation 
$$X^{3}_{1, \alpha,\theta}= \psi_{1}(X^{3}_{1, \alpha,\theta}) \cup \psi_{2}(X^{3}_{1, \alpha,\theta}),$$
where $\psi_{1}$ and $\psi_{2}$ are the maps in \eqref{eq:third case}. 

\medskip

Surprisingly, $X^{3}_{1, \alpha,\theta}$ is not parametrized by a set of ``revolving" sequences but by what we call ``alternating" sequences.
\begin{definition}
A sequence $(\delta_{1},\delta_{2},\dots)\in\Delta_{\theta}^{\NN}$ satisfies {\em the Alternating Condition (AC)}, if 
\begin{enumerate}
\item $\delta_1$ is free to choose, 
\item If $\delta_{1}=\delta_{2}=\dots =\delta_{k}=0$, then  $\delta_{k+1}$ is free to choose,
\item Otherwise, $\delta_{k+1}=0$ or \begin{align*}
\delta_{k+1}=\begin{cases} 	
(e^{+i\theta})\delta_{j_0(k)}, &\mbox{if $N_{j_0(k)}$ is odd},\\ 
(e^{-i\theta})\delta_{j_0(k)}, &\mbox{if $N_{j_0(k)}$ is even}, 
\end{cases}
\end{align*}
where $j_{0}(k):=\max \{j \le k: \delta_{j}\not= 0 \}$ and ${N_{j_{0}(k)}:= \# \{j \leq j_0(k):  \delta_{j} \not= 0\}.}$ 
\end{enumerate}
\end{definition}

Roughly speaking, $N_{j_{0}}(k)$ is the number of times until $j_{0}(k)$ that $\delta_j$ is on the unit circle. Notice that any $\delta_k \not= 0$ must alternate between two values on the unit circle. For example, the following sequence satisfies the AC:
 $$0 \to 0 \to 1 \to e^{i \theta} \to 0 \to 1 \to 0 \to e^{i \theta} \to 0 \to \ldots.$$

\medskip

Define $W^{A}_{\theta}$ as the set of all alternating sequences with parameter $\theta$: 
$$W^{A}_{\theta}:=\{(\delta_{1},\delta_{2},\dots )\in \Delta_{\theta}^{\NN}:(\delta_{1},\delta_{2},\dots )\mbox{ satisfies the AC} \},$$
and for a given $\alpha \in \CC$ such that $|\alpha|<1$, define 
\begin{equation*}
X^{3}_{\alpha,\theta}:=\left\{ \sum_{n=1}^{\infty} \delta_{n} \prod_{j=1}^{n}\xi_j: (\delta_{1}, \delta_{2},\dots )\in W^{A}_{\theta} \right\},
\end{equation*}
where $\xi_1=\alpha$ and 
\begin{equation} 
\xi_{j+1}=\begin{cases}
                \xi_j
                        & \qquad \mbox{ if } \delta_{j}=0, \\
                \overline{\xi_j} 
												& \qquad \mbox{ if } \delta_{j} \not= 0,
        \end{cases}
\label{eq:xi-delta-dependence}
\end{equation}
for $j > 0$. 
\medskip

Four examples of $X^3_{\alpha, \theta}$ are shown in Figure~\ref{figure4}. Notice that $X^3_{\alpha, \theta}$ has a significant difference from $X_{\alpha, \theta}$ and $X^2_{\alpha, \theta}$: the $\prod_{j=1}^{n}\xi_j$ term found in $X^3_{\alpha, \theta}$ depends on the behavior of the sequence $(\delta_1, \delta_2, \cdots, \delta_n)$, while the products in $X_{\alpha, \theta}$ and $X^2_{\alpha, \theta}$ do not depend on that sequence.

\begin{figure}[h] 
\begin{center}
\includegraphics[height=5.5cm,width=5.5cm, bb=0 0 400 400]{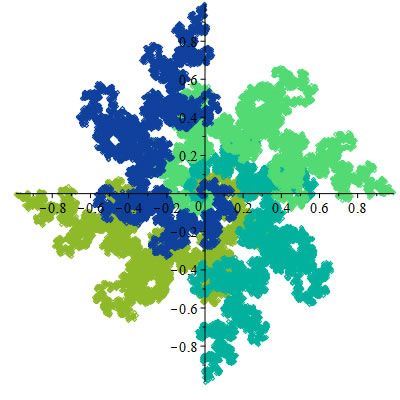} 
\qquad 
\includegraphics[height=5.5cm,width=5.5cm, bb=0 0 400 400]{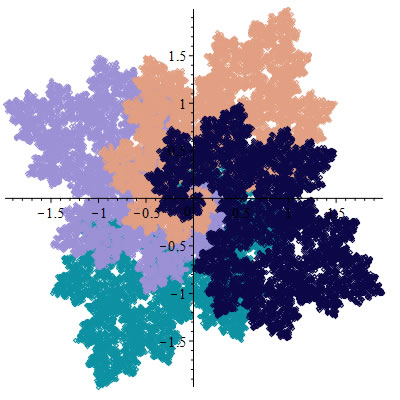} \\

\includegraphics[height=5.5cm,width=5.5cm, bb=0 0 400 400]{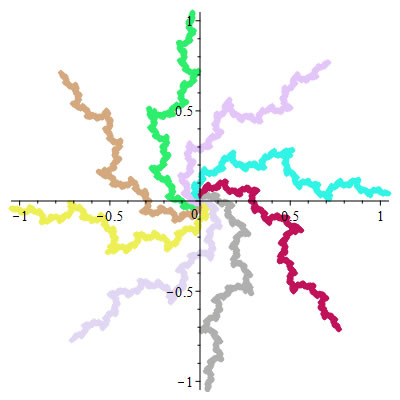} 
\qquad 
\includegraphics[height=5.5cm,width=5.5cm, bb=0 0 400 400]{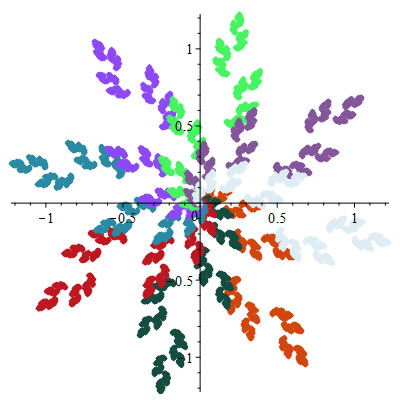}
\end{center}
\caption{$X^{3}_{\alpha,\theta}: (\alpha, \theta)=(\frac{1+i}{2}, \frac{\pi}{2})$ (top left), $(\alpha, \theta)=(\frac{1+i}{2}, -\frac{\pi}{2})$ (top right), 
$(\alpha, \theta)=(\frac{2+i}{4}, \frac{\pi}{4})$ (bottom left), $(\alpha, \theta)=(\frac{2+i}{4}, -\frac{\pi}{4})$ (bottom right)}
\label{figure4}
\end{figure}
\medskip

Let $j_1:=\min\{j: \delta_{j} \not= 0\}$ and define a subset of $X^{3}_{\alpha,\theta}$ as follows.
\begin{equation*}
  X^{3}_{1, \alpha,\theta}=\left\{\sum_{n=1}^{\infty} \delta_{n} \prod_{j=1}^{n}\xi_j : \delta_{j_1}=1, (\delta_{1}, \delta_{2},\dots )\in W^{A}_{\theta} \right\}, \label{xalpha3}
\end{equation*}

A straightforward modification of the proof of Lemma~\ref{lemma:X1-closed} gives the following Lemma~\ref{lemma:X3-closed} and Corollary~\ref{corollary:X3-closed}.

\begin{lemma} \label{lemma:X3-closed}
$X^{3}_{1,\alpha,\theta}$ is a closed set.
\end{lemma}

\begin{corollary}\label{corollary:X3-closed}
$X^{3}_{\alpha,\theta}$ is a closed set.
\end{corollary}

\begin{proposition} \label{prop:X3-set-equation}
$X^{3}_{1, \alpha,\theta}$ satisfies the set equation 
$$X^{3}_{1, \alpha,\theta}= \psi_{1}(X^{3}_{1, \alpha,\theta}) \cup \psi_{2}(X^{3}_{1, \alpha,\theta}), $$
where $\{\psi_1, \psi_2 \}$ is the IFS from \eqref{eq:third case}.
\end{proposition}

\begin{proof}
Let $x=\sum_{n=1}^{\infty}\delta_n \prod_{l=1}^{n}\xi_l \in X^{3}_{1, \alpha, \theta}$, where $(\xi_l)$ depends on $(\delta_n)$ as in \eqref{eq:xi-delta-dependence}. If $\delta_1=0$, set $\delta_j':=\delta_{j+1}$ and $\xi_j':=\xi_{j+1}$ for $j=1,2,\dots$. Then $(\delta_j')$ satisfies the alternating condition with its first nonzero digit equal to 1, and $(\xi_j')$ depends on $(\delta_j')$ as in \eqref{eq:xi-delta-dependence}, with first term $\alpha$. So 
$$x=\alpha \sum_{j=1}^{\infty}\delta_j' \prod_{l=1}^{j}\xi_l'\in \psi_1(X^{3}_{1, \alpha, \theta}).$$

If $\delta_1=1$, set $\delta_j':=e^{i\theta}\overline{\delta_{j+1}}$ and $\xi_j':=\overline{\xi_{j+1}}$ for $j=1,2,\dots$. Since the second nonzero digit of $(\delta_n)$ is $e^{i\theta}$, the sequence $(\delta_j')$ satisfies the alternating condition with its first nonzero digit equal to 1, so 
$$x=\alpha+\alpha e^{i\theta}\sum_{j=1}^\infty\ \overline{\delta_j'} \prod_{l=1}^{j}\overline{\xi_l'} \in \psi_2(X^{3}_{1, \alpha, \theta}).$$
Thus $X^{3}_{1, \alpha,\theta}\subset \psi_{1}(X^{3}_{1, \alpha,\theta}) \cup \psi_{2}(X^{3}_{1, \alpha,\theta})$. The reverse inclusion follows analogously.
\end{proof}

Since the set equation $X=\psi_1(X)\cup \psi_2(X)$ has a unique nonempty compact solution, Theorem \ref{th:main3} follows immediately from Lemma~\ref{lemma:X3-closed} and Proposition \ref{prop:X3-set-equation}.

\begin{theorem}
\label{th:main3}
Let $K^{3}_{\alpha,\theta}$ be the self-similar set generated by the IFS from \eqref{eq:third case}. Then 
\begin{equation*}
\bigcup_{k=0}^{p-1}(e^{i \theta})^k K^{3}_{\alpha,\theta}=X^3_{\alpha, \theta}.
\end{equation*}
\end{theorem}

\begin{remark}
{\rm We originally found the results of Theorems~\ref{th:main1}, \ref{th:main2} and \ref{th:main3} using a technique similar to the one in \cite{Kawamura-2002} (considering the unique solution of a functional equation analogous to \eqref{eq:kiko}). However, to avoid unnecessary technicalities, we have chosen to include only the more direct proofs using Propositions \ref{prop-X1}, \ref{prop:X2-set-equation} and \ref{prop:X3-set-equation}. For the details of the original proofs, see \cite{Kawamura-2019}.
}
\end{remark}

\section*{Acknowledgment}

This research was done mainly during a visit to Utrecht University, Netherlands.
The first author is grateful to Prof.~K.~Dajani for her warm hospitality and encouragement.
Also, we appreciate Prof.~A.~Vince for pointing out the simpler proof of Theorem~\ref{th:main1} and a referee for providing us a proof of 
Lemma~\ref{lemma:X1-closed}. Lastly, we greatly appreciate Prof.~P.~Allaart for his helpful comments and suggestions in preparing this paper.

\end{document}